\documentclass[12pt,oneside]{amsart}
\usepackage[margin=1in]{geometry}

\usepackage{graphicx}
\usepackage{amssymb}
\usepackage{tikz}
\usepackage{lineno}
\usepackage{latexsym,epsfig,enumerate,multicol,wasysym}
\usepackage{amsmath,amsthm,amscd,amssymb}
\usepackage{latexsym}
\usepackage{booktabs}
\usepackage[colorlinks,citecolor=red,pagebackref,hypertexnames=false]{hyperref}
\hypersetup{%
  colorlinks = true,
  citecolor=red,
  linkcolor  = red
}

\usepackage{ulem}
\usepackage{soul}
\usepackage{multirow}
\usepackage{rotating}
\usepackage{pdflscape}
\usepackage{float}
\usepackage{placeins}
\restylefloat{table}
\numberwithin{equation}{section}
\theoremstyle{plain}
\newtheorem{theorem}{Theorem}[section]
\newtheorem{lemma}[theorem]{Lemma}

\newtheorem{proposition}[theorem]{Proposition}

\theoremstyle{definition}
\newtheorem{definition}[theorem]{Definition}

\theoremstyle{remark}
\newtheorem{remark}[theorem]{Remark}

\newtheorem{case[theorem]}{Case}

\date{\today}      
\author{W. Burstein} 
\address{Department of Mathematics, University of Rochester, Rochester, NY}
\email{wburste2@ur.rochester.edu}

\begin{document}

\title[$\Lambda_p$ Style Bounds in Orlicz Spaces Close to $L^2$]{$\Lambda_p$ Style Bounds in Orlicz Spaces Close to $L^2$}  


\begin{abstract}
Let $(\varphi_i)_{i=1}^n$ be mutually orthogonal functions on a probability space such that $\|\varphi_i\|_\infty \leq 1 $ for all $i \in [n]$.  Let $\alpha > 0$.  Let $\Phi(u) = u^2 \log^{\alpha}(u)$ for $u \geq u_{0}$, and $\Phi(u) = c(\alpha) u^2$ otherwise.  $u_0 \geq e$ and $c(\alpha)$ are constants chosen so that $\Phi$ is a Young function, depending only on $\alpha$.  Our main result shows that with probability at least $1/4$ over subsets $I$ of $[n]$, where $I$ is constructed by choosing each index of $[n]$ independently from a Bernoulli distribution, the following holds: $|I| \geq \frac{n}{e \log^{\alpha+1}(n)} $ and for any $a \in \mathbb{C}^n$,
$$
\left \|\sum_{i \in I} a_i \varphi_i \right \|_{\Phi} \leq K(\alpha) \log^{\frac{\alpha}{2}}(\log n) \cdot \|a\|_2.
$$
$K(\alpha)$ is a constant depending only on $\alpha$.
\vskip 0.125in
In the main Theorem of \cite{Ryou22}, Ryou proved the result above to a constant factor, depending on $p$ and $\alpha$, when the Orlicz space is a $L^p(\log L)^{p\alpha}$ space for $p > 2$ where $|I| \sim \frac{n^{2/p}}{\log^{2 \alpha /p}(n)}$.  However, their work did not extend to the case where $p=2$, an open question in \cite{Iosevich25}.  Our result resolves the latter question up to $\log \log n$ factors.   Moreover, our result sharpens the constants of Limonova's main result in \cite{Limonova23} from a factor of $\log n$ to a factor of $\log \log n$, if the orthogonal functions are bounded by a constant.  In addition, our proof is much shorter and simpler than the latter's.
\vskip 0.125in
Finally, to complement our main result, we give a probabilistic lower bound (subsets of $[n]$ are selected by a Bernoulli distribution over $[n]$'s indices) that matches our main result's upper bound.  Hence, if our bounds can be improved, the improvement cannot use the common probabilistic technique: subset selection by a Bernoulli distribution on indices of $[n]$.
\end{abstract}  

\maketitle

\tableofcontents

\section{Introduction}
Let $Z$ be a random variable on a probability space, $(\Omega, P)$, and let $q \geq 1$.  We write the norm of $Z \in L^{q}(\Omega)$ as,
$$
\|Z\|_q := \left ( 
\mathbb{E}\left [|Z|^q \right]
\right )^{1/q}.
$$
Let $(a_1,...,a_n) \in \mathbb{C}^n$.
The sequence $(\varphi_i : \Omega \rightarrow \mathbb{C})_{i=1}^n$ induces the function $f = \sum_{i=1}^n a_i \varphi_i$.  We write $a_f:=(a_1,...,a_n)$ as the coefficients that parametrize $f$.
Let $\chi_t(x) = \exp(2 \pi i t x)$.  Let $\Omega := [0, 1]$ and $p > 2$.  A subset $S$ of $\mathbb{Z}$ is called a $\Lambda(p)$-set if
$$
\|f\|_{L^p(\Omega)} \leq C(p) \cdot \|a_f\|_2
$$
for all $f \in \text{span}(\chi_t)_{t \in S}$, where $C(p)$ is a constant depending only on $p$. 
\vskip 0.125in
In \cite{Rudin60}, for each even integer $n > 2$, Rudin explicitly constructs $\Lambda(n)$-sets which are not $\Lambda(n + \varepsilon)$-sets for all $\varepsilon > 0$.  The latter problem becomes more challenging when $n$ is odd since it is difficult to construct explicit examples.  In \cite{Bourgain89}, Bourgain overcomes these difficulties by using the probabilistic method to show generically, for any $p > 2$, there exist $\Lambda(p)$-sets which are not $\Lambda(p+ \varepsilon)$-sets for all $\varepsilon > 0$.  To solve the latter problem, Bourgain proved the following.
\begin{theorem}
\label{BourgainMain}
Let $(\varphi_i : \Omega \longrightarrow \mathbb{C})_{i=1}^n$ be mututally orthogonal functions on a probability space, $(\Omega, P)$, where $\|\varphi_i\|_{\infty} \leq 1$.  Let $p > 2$.  Then there exists a subset $S \subset [n]$ such that $|S| \geq n^{2/p}$ and 
$$
\|f\|_{L^p(\Omega)} \leq C(p) \cdot \|a_f\|_2
$$
for all $f \in \text{span}( \varphi_i)_{i \in S}$, where $C(p)$ is a constant depending only on $p$.
\end{theorem}
Bourgain's proof of Theorem \ref{BourgainMain} is extremely sophisticated, using a clever combination of results built-up from scratch in decoupling theory, bounds on metric entropy, dyadic decomposition, and the probabilistic method (see \cite{Tao20}).  As groundbreaking and transparent as Bourgain's proof of Theorem $\ref{BourgainMain}$ is, the proof relies on the properties of the function $\phi(z)=z^p$ and is broken up into the cases $2 < p \leq 3$, $3 < p \leq 4$, and $p >4$.  In \cite{Talagrand95}, Talagrand presents a more conceptual proof of Theorem \ref{BourgainMain}.  Talagrand's proof relies on an abstract definition called $2$-smoothness of norm and notes that $L^p$ is $2$-smooth.  This more general framework eliminated the case by case analysis described above and additionally proves the result in $L^{p,1}$ spaces, something which cannot be directly proved using Bourgain's original approach.  In his textbook, \cite{Talagrand21}, Talagrand has since developed a simpler conceptual framework than in \cite{Talagrand95} using the notion of $2$-convexity.  The $\Lambda(p)$ problem is the main application of the latter conceptual framework. In this work, our main result applies the latter abstract machinery to solve the problem. 
\vskip 0.125in
This work gives analogues to Theorem \ref{BourgainMain} in Orlicz spaces.  We present terminology needed to define Luxemburg norms, which are equivalent to Orlicz norms.  
\begin{definition}
Assume $\Phi: [0, \infty) \longrightarrow [0, \infty)$ has the following proprerties: $\Phi(0)=0$ and $\Phi$ is convex, increasing, and unbounded.  If $\Phi$ satisfies the latter properties, it is called a Young function.  
\end{definition}
Let $f$ be a function on a probability space, $(\Omega, P)$. We define the Luxemburg norm, $\|\cdot \|_{\Phi}$, as follows:
$$
\|f\|_{\Phi} = \inf \left\{ k > 0 : \mathbb{E}\left [\Phi \left (\frac{|f|}{k} \right) \right] \leq 1  \right\}.
$$
The space of all random variables, which are finite on $\| \cdot\|_{\Phi}$, is called an $L^{\Phi}$ space.  $\| \cdot \|_{\Phi}$ defined in this way yields a norm on $L^{\Phi}$.
\begin{definition}
A Young function, $\Phi$, is called a nice Young function if
$$
\lim_{u \rightarrow 0} \frac{\Phi(u)}{u}=0 \text{ and }
\lim_{u \rightarrow \infty} \frac{\Phi(u)}{u} = \infty.
$$
\end{definition}
\vskip 0.125in
Let $\mathcal{C}$ be the class of nice Young functions satisfying the assumptions of Theorem 1.5 in \cite{Ryou22}.  Since we only care about Zygmund spaces, a special case of $\mathcal{C}$ in Theorem \ref{Zygmund} below, we do not list the specific assumptions $\mathcal{C}$ satisfies as the assumption list is long and complicated.  On the probability space $[0,1]$ with Lebesgue measure, if $\Phi \in \mathcal {C}$, it was shown in Theorem 1.5 of \cite{Ryou22}, there exists $J \subset [n]$ such that $|J| \geq \Phi^{-1}(n)^2$ and
\begin{align}
\label{DongeunTheorem}
\|f\|_{\Phi} \leq K \|a_f\|_2
\end{align}
for all $f \in \text{span}( \chi_{t})_{t \in J}$, where $K$ is a constant essentially depending on $\Phi$.  It should be noted that $\Phi(u) = u^p$ is in $\mathcal{C}$.  Thus Theorem \ref{BourgainMain} is recovered.
\vskip.125in 
Important to this work, we have the special case of equation \ref{DongeunTheorem} where we define the nice Young function in $\mathcal{C}$ as
\[ \Phi_{p, \alpha}(u) = \begin{cases} 
      cu^2 & u\leq u_0 \\
      u^p \log^{\alpha p}(u) & u > u_0 
   \end{cases},
\]
where $p > 2$ and $\alpha > 0$, see example 1.10 in \cite{Ryou22}.  Note that $c$ and $u_0 \geq 1$ can be chosen so that $\Phi_\alpha$ is a nice Young function.  Thus, from \cite{Ryou22}, we obtain the next result.
\begin{theorem}
\label{Zygmund}
Let $p>2$ and $\alpha > 0$.  There exists $J \subset [n]$ such that $|J| \geq 
\frac{n^{2/p}}{\log^{\alpha}(n)}$ and 
$$
\|f\|_{\Phi_{p,\alpha}} \leq C(p, \alpha) \cdot \|a_f\|_2
$$
for all $f \in \text{span}( \chi_t)_{t \in J}$, where $C(p, \alpha)$ is constant depending only on $\alpha$ and $p$.  
\end{theorem}
$L^{\Phi_{p, \alpha}}$ is known as a $L^{p}(\log L)^{p\alpha}$ space or a Zygmund space.  For functions on $\left (\mathbb{Z}/N\mathbb{Z} \right)^d$, it is mentioned in \cite{Iosevich25} that the result of Theorem \ref{Zygmund} would be particularly useful for $p=2$ in applications to signal recovery (see also \cite{BIMN2025, Iosevich24}), which the authors left as an open problem.  Our main result is a proof of the $p=2$ case, a case not covered in \cite{Ryou22}.
\vskip 0.125in
Bourgain originally proved Theorem \ref{TalMain}, below, in an unpublished manuscript, and then Talagrand gave a more conceptual proof of it in \cite{Talagrand98}.  Talagrand's proof uses a similar conceptual framework as in \cite{Talagrand95}, replacing the abstract definition of $2$-smoothness with a 'local' generalization called condition $H(C)$.  
\begin{theorem}
\label{TalMain}
Let $(\varphi_i)_{i=1}^n$ be a sequence of orthogonal functions on a probability space, $(\Omega, P)$, such that $\|\varphi_i\|_{\infty} \leq 1$ and let $\lambda = \inf_{1 \leq i \leq n} \|\varphi_i\|_2$.  Then there is a constant, $c_{\lambda} > 0$, such that for most subsets $I$ of $[n]$, with cardinality $\leq c_\lambda n$, we have for each number $(a_i)_{i \in I}$ that
$$
\left \|\sum_{i \in I} a_i \varphi_i \right \|_1 
\geq\frac{\lambda^2}{K\sqrt{\log n \log \log n}}
\left (\sum_{i \in I} a_i^2 \right)^{1/2}.
$$
\end{theorem}

\vskip 0.125in
Similar results to Theorem \ref{TalMain} were proved in \cite{Guedon07} and \cite{Guedon08}, both using conceptual frameworks similar to those in \cite{Talagrand98}. The main intermediate result and application of the abstract machinery in \cite{Guedon08} is below.
\begin{proposition}
\label{GuedonResult}
There exists a universal constant, $K > 0$, such that the following holds. Suppose $q \in (1,2)$ and $(\varphi_i)_{i=1}^n$ is an orthonormal system on a probability space, $(\Omega, P)$, uniformly bounded by $L$.  Then there exists $I \subset [n]$ such that $|I| \geq n-k$ and for every $f \in span(\varphi_i)_{i \in I}$,
$$
\|f\|_2 \leq \frac{K}{(q-1)^{5/2}} L \sqrt{n/k} \sqrt{\log k}\|f\|_q.
$$
\end{proposition}

\vskip 0.125in

Let $\alpha > 1/2$.  Define the Orlicz function, see \cite{Kashin20, Limonova23}, as
$$
G_{\alpha}(u) = u^2 \frac{\ln^{\alpha}(e+|u|)}{\ln^{\alpha}(e +1/|u|)}.
$$

The following Theorem, Theorem 2 of \cite{Limonova23}, is below, where the orthogonal functions are assumed to be bounded in $p$ norm by a constant for a fixed $p > 2$.  We improve the constants in Theorem \ref{RussiaResult} significantly if the orthogonal functions are uniformly bounded by a constant.  In addition, our proof is much shorter and simpler (see Theorem \ref{main}).
\begin{theorem}
\label{RussiaResult}
Let $\alpha > 3/2$, $\rho > 2$, and $p > 2$.  Consider mutually orthogonal functions, $(\varphi_i)_{i=1}^n$, such that $\| \varphi_i \|_p \leq 1$ for all $i \in [n]$.  Set $\beta = \max(\alpha/2-\rho/4,1/4)$.  Let $I \subset [n]$ be a random set such that the probability of choosing an index, $i$, is $\log^{-\rho}(n+3)$ for each $i \in [n]$.  Then with high probability, for a set $I \subset [n]$, the following holds for all $f \in \text{span}(\varphi_i)_{i \in I}$:
\begin{equation}
\label{eq:Limonova}
\|f\|_{G_{\alpha}} \leq C(\alpha, p, \rho) \log^{\beta+\frac{1}{2}}(n+3)\|a_f\|_2,
\end{equation}
where  $C(\alpha, p, \rho)$ is a constant depending only on $\alpha, p$, and $\rho$.
\end{theorem}
\begin{remark}
We improve Limonova's main result, Theorem \ref{RussiaResult}, when the orthogonal functions are bounded uniformly by a constant, a common use case (see \cite{Iosevich25}).  We compare Limonova's result to our main result, Theorem \ref{main}, in the uniformly bounded setting.  First, it should be noted that in Theorem \ref{RussiaResult}, the density is $\frac{1}{\log^{\rho}(n)}$, and in our main result, we obtain results, only, in the case $\rho=\alpha+1$.  Our result has a $\log^{\frac{\alpha}{2}}(\log n)$ factor in equation \ref{eq:main} compared to a  $\log^{\beta+\frac{1}{2}}(n)$ factor in equation \ref{eq:Limonova}.  Our proof of Theorem \ref{main} is much shorter and more conceptual than the proof Theorem \ref{RussiaResult}, as we build it on top of Talagrand's abstract Theorem, Theorem \ref{TalTech}.  Conversely, the proof of Theorem \ref{RussiaResult} uses similar techniques and lower-level details of Bourgain's proof, of Theorem \ref{BourgainMain}, which leads to a longer and more complicated proof.
\end{remark}

\vskip 0.125in

It is worth mentioning the interesting result, Theorem 1 of \cite{Kashin20}, which we state below. 
\begin{theorem}
\label{thm:kashin20}
Let $\alpha > 0$ and $\rho >0$ be fixed.  Let $(\varphi_i)_{i=1}^n$ be an orthogonal system such that $\|\varphi_i\|_{\infty} \leq 1$ for all $i \in [n]$.  Let $(\zeta_i)_{i=1}^n$ be i.i.d. and $\{0,1\}$-valued.  Then with high probability, for a random set $\Lambda = \Lambda_\omega \subset [n]$ generated by $(\zeta_i(\omega))_{i=1}^n$ with $\delta = \mathbb{E} \zeta_i = \log(n+3)^{-\rho}$, for all $i \in [n]$, it follows that for all $\|a\|_{\infty} =1$ such that $\text{supp}(a) \in \Lambda$,
$$
\left \|\sum_{i \in \Lambda} a_i
\varphi_i 
\right \|_{G_{\alpha}}
\leq K(\alpha,\rho) \cdot |\Lambda|^{\frac{1}{2}}
\cdot \left (\log^{\alpha/2-\rho/4}(n + 3)
+1.
\right ),
$$
where $K(\alpha, \rho)$ is a constant depending on $\alpha$ and $\rho$.
\end{theorem}
\begin{remark}
In the case of Theorem \ref{thm:kashin20} when $\delta = \log^{-2 \alpha}(n+3)$ and $|a_i| = \frac{1}{|\Lambda|^{1/2}}$, we have that 
$$
\left \|\sum_{i \in \Lambda} a_i
\varphi_i 
\right \|_{G_{\alpha}}
\leq K(\alpha, \rho).
$$
\end{remark}
\begin{remark}
We thank the referees for asking if the techniques of our main Theorem, Theorem \ref{main}, could give a proof of Theorem \ref{thm:kashin20}.  Theorem \ref{thm:kashin20} uses techniques from \cite{Bourgain89}, so it would be desirable to have a simpler and more conceptual proof, analogous to how the proof of our main Theorem, Theorem \ref{main}, is more conceptual than the proof of Theorem \ref{RussiaResult}.  The main technical ingredient of our main Theorem, Theorem \ref{TalTech} of \cite{Talagrand21}, unfortunately only works for operators on $l_n^q$ when $1 < q \leq 2$, and Theorem \ref{thm:kashin20} would require $q= \infty$.  We leave the problem of finding a more conceptual proof of Theorem \ref{thm:kashin20}, perhaps using an abstract framework similar to those in \cite{Talagrand21}, as future work.
\end{remark}

\vskip 0.25in

\section{Results}
The main Theorem of \cite{Ryou22} proves our main result, without $\log \log n$ factors, when the Orlicz space is a $L^p(\log L)^{p\alpha}$ space for $p > 2$ (see Theorem \ref{DongeunTheorem}).  However, their work did not extend to the case where $p=2$.  The latter is an open problem introduced in \cite{Iosevich25}, which we prove up to $\log \log n$ factors as our main result.  First, we list some definitions and Theorems in general Banach spaces, which will be applied to prove our main result, Theorem \ref{main}.  Second, we prove our main result.  Finally, we give a probabilistic lower bound (subsets of $[n]$ are selected by a Bernoulli distribution over $[n]$'s indices) matching the bounds of equation \ref{eq:main} in our main result.  Hence, if our bounds can be improved, the improvement cannot use the common probabilistic technique: subset selection by a Bernoulli distribution on indices of $[n]$.  In proofs, constants can vary from line to line. 
\subsection{Definitions and Theorems in Banach Spaces}
First some definitions from \cite{Talagrand21}.  Let $X$ be a Banach space and consider $n$ linearly independent elements of $X$, $x_1,..,x_{n}$. 
 Consider the space $l_{n}^{q}$ on $\mathbb{C}^{n}$ with canonical basis $(e_i)_{i=1}^n$.  We define,
$$
U : l_{n}^q \rightarrow X
$$
to be the operator sending $e_i$ to $x_i$.  In this work, we will be interested in the case where $q=2$.  Let $p$ be the conjugate of $q$. 
 Given a number $C > 0$, we can define the norm, $\| \cdot \|_C$, on $X$ as follows (see \cite{Talagrand21} page 632).  $\| \cdot \|_C$ is induced by defining the unit ball of the dual norm to be,
\begin{align}
\label{eq:C}
X_{1,C}^{*} :=
\left \{ 
x^{*} \in X^{*} : \|x^{*}\|_{X^*}\leq 1,
\sum_{i=1}^n \left |x^{*}(x_i) \right |^p \leq C
\right \}.
\end{align}
We denote $\|U\|_C$ to be the operator norm of $U$ when $X$ is given the norm $\|\cdot\|_C$.
\vskip 0.125in
We recall the definition of $p$-convex (see page 112 of \cite{Talagrand21}).  
\begin{definition}
$X$'s norm, $\| \cdot \|$, is $p$-convex if there is a number, $\eta > 0$, such that for all $x,y \in X$, where $\|x\|, \|y\| \leq 1$, then
\begin{align}
\label{eq:pconvex}
\left \|\frac{x+y}{2} \right \| \leq 1 - \eta \|x-y\|^p.
\end{align}
\end{definition}
\begin{remark}
We note that if $X = L^p$ and $p \leq 2$, then $\| \cdot \|_p$ is $2$-convex (see \cite{Lindenstrauss79}). 
\end{remark}
\vskip 0.125in
We define a random subset, $J$, of $[n]$ by taking an i.i.d. vector of Bernoulli random variables, $(\zeta_i)_{i=1}^n$, where,
\begin{align}
\label{eq:ber}
\Pr( \zeta_i = 1) = \delta \text{; }
\Pr( \zeta_i = 0) = 1 - \delta.
\end{align}
Thus,
$$
J = \left \{ 
i \in [n] : \zeta_i =1
\right \}.
$$
We can define $U_J$ to be the operator from $l_J^q \rightarrow X$ given by $U(e_i) = x_i$ where $i \in J$.  
\vskip 0.125in
We now state a technical Lemma, which our main result is an application of (see Theorem 19.3.1, pg 632 of \cite{Talagrand21}).
\begin{theorem}
\label{TalTech}
Consider $1 <  q \leq 2$ and its conjugate exponent $p \geq 2$.  Consider a Banach space $X$ such that $X^*$ is $p$-convex with corresponding $\eta >0$.  Then there exists a constant, $K(p, \eta)$, depending only on $p$ and $\eta$ with the following property.  Consider elements, $x_1, \dots, x_n \in X$, and set $S = \max_{1 \leq i \leq n} \|x_i\|$. 
 Denote the operator $U: l_n^q \rightarrow X$ given by $U(e_i) = x_i$.  Consider a number $C > 0$ and define $B = \max(C, K(p, \eta) \cdot S^p \log n)$.  Assume that for some number, $\varepsilon > 0$,
 $$
 \delta \leq \frac{S^p}{B \varepsilon n^{\varepsilon}} \leq 1.
 $$
 Consider the i.i.d. random vector, $(\zeta_i)_{i=1}^n$, and random set, $J$, defined in equation \ref{eq:ber}.  Then,
 $$
 \mathbb{E}\|U_J\|_C^p \leq 
 K(p, \eta)\frac{S^p}{\varepsilon}.
 $$
\end{theorem}

\subsection{Main Result}
Throughout this section we assume that $ (\varphi_i)_{i=1}^n$ are mutually orthogonal and non-zero functions on a probability space, $(\Omega, P)$, and that $\|\varphi_i\|_{\infty} \leq 1$.  Note that $ (\varphi_i)_{i=1}^n$ are linearly independent since they are non-zero and orthogonal.  Set $E := L^{p_1}(\Omega)$ with norm $\| \cdot \| := \| \cdot\|_{p_1}$ where $p_1 > 2$ will be selected later.  Denote the dual space by $E^{*}$.  Let $q_1$ be $p_1$'s conjugate.  Hence, $\|\cdot\|^{*}$ on $E^{*}$ is $\| \cdot \|_{q_1}$ since $E^{*} = L^{q_1}(\Omega)$.  Hence, $\|\cdot\|^{*}$ is $2$-convex since $q_1 \leq 2$ in $L^{q_1}(\Omega)$ (see equation \ref{eq:pconvex}). 
\vskip 0.125in
We state a technical Lemma (see Lemma 19.3.11 of \cite{Talagrand21}).  Define the $\| \cdot \|_C$ norm on $E$ as induced in equation \ref{eq:C}.
\begin{lemma}
\label{hahnbanach}
If $f \in E$ satisfies $\|f\|_C \leq 1$, then $f \in \text{conv}(\mathcal{C}_1 \cup \mathcal{C}_2)$ where $\mathcal{C}_1
= \left\{ g \in E : \|g\| \leq 1 \right\}$ and $\mathcal{C}_2 = 
\left \{
\sum_{i=1}^n \beta_i \varphi_i : \| \beta \|_2^2 \leq \frac{1}{C}
\right \}.
$
\end{lemma}
\vskip 0.125in
\begin{lemma}
\label{abstractLemma}
Let $\alpha > 0$ and $p_1 > 2$ with conjugate $q_1$.  Define, $U_{J_\alpha} : l_{J_\alpha}^2 \rightarrow E$, where $J_{\alpha}$ is defined by an i.i.d. Bernoulli random vector with parameter $\delta$ as in equation \ref{eq:ber}.  Define $S:=\max_{1 \leq i \leq n}\|\varphi_i\|_{p_1}$ and set $\delta = \frac{1}{e \log^{\alpha+1}(n)}$.  Set $C_{\alpha} = 
S^2\log^{\alpha + 2}(n)$.  Then,
$$
\mathbb{E}\|U_{J_{\alpha}}\|_{C_\alpha} \leq K(\eta) \cdot S \sqrt {\log n}
$$
where $\eta > 0$ is a parameter that depends on the $2$-convexity of $\|\cdot\|_{q_1}$ in $E^{*}=L^{q_1}(\Omega)$ and $K(\eta)$ is a constant depending only on $\eta$. 
\end{lemma}
\begin{proof}
Set $X:=E$, $p := 2$, $q := 2$, $x_i := \varphi_i$ for each $i \in [n]$, and $\varepsilon := \frac{1}{\log n}$.  Since $\|\cdot\|^{*} = \| \cdot \|_{q_1}$ is $2$-convex,  we have aligned our notation with Theorem \ref{TalTech}.  The parameter, $B=\max(C_{\alpha}, K(\eta) S^2 \log(n)) = C_{\alpha}$ by definition of $C_{\alpha}$. Hence,
$$
D=\frac{S^2}{B \varepsilon n^{\varepsilon}}=
\frac{S^2 \log n}{S^2\log^{\alpha+2}(n) e}= 
\frac{1}{\log^{\alpha+1}(n) e} = \delta.
$$
The assumptions of Theorem \ref{TalTech} are satisfied and we get that,
$$
\mathbb{E}\|U_{J_{\alpha}}\|_{C_{\alpha}}^2 \leq \frac{S^2 K(\eta)}{\varepsilon}\leq S^2 K(\eta)\log n.
$$
Noting the bound, $(\mathbb{E} Z)^2 \leq \mathbb{E} (Z^2)$, for a random variable $Z$, we complete the proof.
\end{proof}
\vskip 0.125in
We now go back to a more concrete notation since most of our abstract applications are complete.
\begin{lemma}
\label{genericstatement}
Let $\alpha >0$ and $p_1 > 2$.  Choose $J_{\alpha} \subset [n]$ depending on a Bernoulli i.i.d. random vector with parameters, $C_{\alpha}$ and $\delta$ as in Lemma \ref{abstractLemma}.  Then with probability at least $1/4$ over $J_{\alpha} \subset [n]$, the probability distribution of Lemma \ref{abstractLemma},  the following statement holds, $|J_{\alpha}| \geq \frac{n}{e \log^{\alpha+1}(n)}$ and for any $a \in \mathbb{C}^n$, setting $f = \sum_{i \in J_{\alpha}} a_i \varphi_i$,
$$
\|f\|_{C_{\alpha}} \leq K \cdot S\sqrt{\log n}\cdot \|a\|_2,
$$
where $S = \max_{1 \leq i \leq n}\|\varphi_i\|_{p_1}$ and $K > 0$ is a constant.
\end{lemma}
\begin{proof}
Set $K = K(\eta)$ to be a constant since $\eta$ is a constant depending only on $2$.  By Lemma \ref{abstractLemma} we have that,
$$
\mathbb{E}\|U_{J_\alpha}\|_{C_\alpha} \leq K \cdot S \sqrt{\log n}.
$$
By Markov's inequality,
$$
\Pr \left (\|U_{J_\alpha}\|_{C_{\alpha}} \geq 4K \cdot S \cdot \sqrt{\log n} \right) \leq 1/4.
$$
Hence, with probability at most $1/4$,
$$
\|f\|_{C_\alpha} \geq 4 \cdot K \cdot  S \cdot \sqrt{\log n} \|a\|_2,
$$
for some $f$ as defined in the Lemma statement.
$|J_\alpha|$ is distributed as a binomial random variable with parameter, $(n,\delta)$, where $\delta = \frac{1}{e \log^{\alpha+1}(n)}$.  Thus $|J_{\alpha}|$ has mean, $\frac{n}{e \log^{\alpha+1}(n)}$.  Hence, from basic facts from probability theory we have,
$$
\Pr \left (|J_{\alpha}| \geq  \frac{n}{e \log^{\alpha+1}(n)} 
\right ) \geq 1/2.
$$
From a union bound we have that,
$$
\Pr\left (
|J_{\alpha}| \geq  \frac{n}{e \log^{\alpha+1}(n)} \text{ and }
\sup_{\|a\|_2 \leq 1}
\left\| \sum_{i \in J_{\alpha}} a_i \varphi_i  
\right\|_{C_{\alpha}} \leq 4 \cdot K \cdot S \cdot \sqrt{\log n}
\right ) \geq 1/4.
$$
\end{proof}
\vskip 0.125in
\begin{theorem}
\label{main}
Let $(\varphi_i)_{i=1}^n$ be mutually orthogonal functions on a probability space, $(\Omega, P)$, such that $\|\varphi_i\|_\infty \leq 1 $ for all $i \in [n]$.  Let $\alpha > 0$.  Let $\Phi(u) = u^2 \log^{\alpha}(u)$ for $u \geq u_{0}$ and $\Phi(u) = c(\alpha) u^2$ otherwise.  $u_0 \geq e$ and $c(\alpha)$ are constants, depending on $\alpha$,  chosen so that $\Phi$ is a Young function.  Then with probability at least $1/4$ over subsets $I$ of $[n]$, the following holds: $|I| \geq \frac{n}{e \log^{\alpha+1}(n)} $ and
\begin{align}
\label{eq:main}
\left \|\sum_{i \in I} a_i \varphi_i \right \|_{\Phi} \leq K(\alpha) \cdot \log^{\frac{\alpha}{2}}(u_0+S^2\log n) \cdot \|a\|_2
\end{align}
for any $a \in \mathbb{C}^n$.  $K(\alpha)$ is a constant depending only on $\alpha$, and $S = \max_{1 \leq i \leq n}\|\varphi_i\|_{p_1}$ where $p_1 - 2 > 0$ is a constant only depending on $\alpha$.
\end{theorem}
\begin{remark}
In Theorem \ref{eq:main}, $S \leq \max_{1 \leq i \leq n}\|\varphi_i\|_{\infty} \leq 1$, so we can formulate the factor more simply as 
$$
K(\alpha) \cdot \log^{\frac{\alpha}{2}}(\log n).
$$
\end{remark}

\begin{remark}
Equation \ref{eq:main} has the following trivial version:
$$
\left\|
\sum_{i=1}^n a_i \varphi_i
\right\|_{\Phi} \leq
K(\alpha) \cdot \log^{\frac{\alpha}{2}}(n) \cdot \|a\|_2.
$$
Indeed, assuming WLOG that $\|a\|_2=1$ and setting $f:=\sum_{i=1}^n a_i \varphi_i$, 
$$
\left\|
f
\right\|_{\Phi} =
\inf \left \{
k > 0: \mathbb{E} \left [\Phi \left (\frac{|f|}{k} \right ) \right]
\leq 1
\right \}
$$
$$
=
\inf \left \{
k > 0 :
\frac{c(\alpha)\mathbb{E} 
\left [
\left |f
\right|^2 \cdot 
\bold{1}
_{\left \{
\frac{|f|}{k} \leq u_0
\right \}}
\right ]
}{k^2}
+
\frac{\mathbb{E} 
\left [
\left |
f
\right|^2 \log^{\alpha}
\left(
\frac{ \left|
f
\right|}{k}
\right )
\cdot
\bold{1}_{
\left \{
\frac{|f| }{k} \geq u_0
\right \}}
\right ]
}{k^2} \leq 1
\right \}
$$
$$
\leq
\inf \left \{
k > 0 :
\frac{c(\alpha)
}{k^2}
+
\frac{\mathbb{E} 
\left [
\left |
f
\right|^2 \log^{\alpha}
\left(
\left|
f
\right|
\right )
\cdot
\bold{1}_{
\left \{
\frac{|f| }{k} \geq u_0
\right \}}
\right ]
}{k^2} \leq 1
\right \}
$$
$$
\leq
\inf \left \{
k \geq 1 :
\frac{c(\alpha)
}{k^2}
+
\frac{\mathbb{E} 
\left [
\left |
f
\right|^2 \log^{\alpha}
\left (
n^{\frac{1}{2}}
\right )
\right ]
}{k^2} \leq 1
\right \}
$$
$$
\leq
\inf \left \{
k \geq 1 :
\frac{c(\alpha)
+\log^{\alpha}
\left (
n^{\frac{1}{2}}
\right )}{k^2} \leq 1
\right \}
$$
$$
= \max
\left (1, \sqrt{
c(\alpha)
+ \frac{\log^{\alpha}(n)}{2^{\alpha}}
} 
\right) 
$$
$$
=
K(\alpha)\log^{\frac{\alpha}{2}}(n),
$$
where the second equality above is by definition, the third inequality is by noting that $\|f\|_2 \leq 1$ and that if $A \subset B$ then $\inf A \geq \inf B$, and the fourth inequality comes from Cauchy-Schwarz inside the $\log$ on $\left |\sum_i a_i \varphi_i(\omega) \right|$ and noting that $\|a\|_2=1$ and $\|\varphi_i\|_{\infty} \leq 1$.  Hence, for any $a \in \mathbb{C}^n$, we can apply the above bound to $\sum_i \frac{a_i}{\|a\|_2} \varphi_i $ and then multiply by $\|a\|_2$.
\end{remark}

\vskip 0.125in

Below is the proof of Theorem \ref{main}.
\begin{proof}
Throughout the proof, $\log := \ln$.
Set $C:= C_\alpha = S^2\log^{\alpha+2}(n)$.  We set $I := J_{\alpha}$ from Lemma \ref{genericstatement} and get
\begin{align}
\label{eq:probabilitystatemett}
\Pr\left (
|I| \geq  \frac{n}{e \log^{\alpha+1}(n)} \text{ and }
\sup_{\|a\|_2 \leq 1}
\left\| \sum_{i \in I} a_i \varphi_i  
\right\|_{C} \leq K \cdot S \sqrt{\log n}
\right ) \geq 1/4.
\end{align}
We assume $I$ satisfies the properties of equation \ref{eq:probabilitystatemett} from now on.
The proof proceeds as in Lemma 4.3 of \cite{Ryou22}.  Let $p_1 > r > 2$ be constants.  By the definition of $\Phi$, when $u \geq 1$,
\begin{align}
\label{eq:2}
\Phi^{'}(u) 
\leq (2 + \alpha) \frac{\Phi(u)}{u}
\leq K(\alpha) u^{r-1}.
\end{align}
Note that the left inequality, above, holds for $u \geq 0$.

Fix $f$, where $f = \sum_{i \in I} a_i \varphi_i$, and assume WLOG that $\|a\|_2=1$.  By the triangle inequality, for any $D > 0$,
$$
\|f\|_{\Phi} \leq \|f \cdot \bold{1}_{|f| \leq D}\|_{\Phi}
+\|f \cdot \bold{1}_{|f| \geq D}\|_{\Phi}
$$
$$
= I + II.
$$

\vskip 0.25in

We need to bound $I$ and $II$; thus, we will need the next set of calculations.  Set $Z  = \Phi(|f \cdot  \bold{1}_{|f| \leq D}|)$, and recall that $Z$ is defined on $(\Omega, P)$.  Recall that $\Phi$ is strictly increasing which implies that its inverse is strictly increasing.  We calculate the following:
$$
\mathbb{E}[Z] = \mathbb{E}\left (\int_{0}^{Z}dt \right )
=\mathbb{E} \int_{0}^{\infty} \bold{1}(t \leq Z) dt
$$
$$
= \int_{0}^{\infty} \mathbb{E} \bold{1}(t \leq Z) dt
= \int_{0}^{\infty} P[ Z \geq t] dt 
= \int_{\Phi(0)}^{\Phi(D)} P[ Z \geq t] dt
$$
$$
\leq \int_{\Phi(0)}^{\Phi(D)} P[\Phi(|f|) \geq t]dt
= \int_{0}^{D} \Phi^{'}(u) P[\Phi(|f|) \geq \Phi(u)]du
$$
\begin{align}
\label{firsteq}
= \int_{0}^{D} \Phi^{'}(u) P[|f| \geq u]du.
\end{align}
The third equality above is from Fubini's Theorem. The fourth equality above is from the identity, $\mathbb{E} \bold{1}(t \leq Z) = P[ Z \geq t]$.  The fifth equality above is from noting that $|f| \cdot \bold{1}_{|f|\leq  D}(\omega) = 0$ if $|f(\omega)| \not \in [0, D]$, and hence, $\Phi(0) \leq Z(\omega) = \Phi(|f \cdot \bold{1}_{|f|\leq D}|)(\omega) \leq \Phi(D)$, taking into consideration that $\Phi$ is increasing and $\Phi(0)=0$.  The sixth inequality comes from the fact that $Z \leq \Phi(|f|)$.  Defining $F(t) = P(\Phi(|f|) \geq t)$, the seventh equality comes from the substitution $t = \Phi(u)$.  The eighth equality comes from the fact that $\Phi$ is increasing, and hence $\Phi(|f|) \geq \Phi(u)$ if and only if $|f| \geq u$.  
\vskip 0.125in
Similarly, setting $Z' = \Phi(|f| \cdot \bold{1}_{|f| \geq D})$ we get,
\begin{align}
\label{eq:3}
\mathbb{E}[Z'] = 
\int_0^{\infty} P[\Phi(|f| \cdot \bold{1}_{|f|\geq D}) \geq t]dt
\end{align}
$$
=
\int_{0}^{\Phi(D)} 
P[\Phi(|f| \cdot \bold{1}_{|f|\geq D}) \geq t]dt+
 \int_{\Phi(D)}^{\infty} P[\Phi(|f| \cdot \bold{1}_{|f|\geq D}) \geq t]dt
$$
$$
=
\int_{0}^{\Phi(D)} 
P[|f| \geq D]dt+
 \int_{\Phi(D)}^{\infty} P[\Phi(|f| \cdot \bold{1}_{|f|\geq D}) \geq t]dt
$$
$$
= \Phi(D) \cdot P[|f| \geq D] + 
 \int_{D}^{\infty} \Phi'(t) \cdot P[|f|\geq t]dt
$$
We note that the third equality above comes from the fact that $P[Z' > 0] = P[Z' \geq t] = P[|f| \geq D]$ for all $t \in (0, \Phi(D))$.

\vskip 0.25in
Using techniques that will be useful to bound $II$, we warm-up by bounding $I$.  After the warm-up, we will get a sharper bound on $I$.  By convexity of $\Phi$ and since $\Phi(0) = 0$, it follows that $\Phi(u/k) \leq \Phi(u)/k$ for $k \geq 1$.
Thus, we bound $I$ and get that,
\begin{align}
\label{eq:4}
I \leq \inf 
\left \{k \geq 1 :
\mathbb{E} 
\left [
\Phi \left ( \frac{  |f| \cdot \bold{1}_{|f| \leq D}
}{k}
\right ) 
\right ]
\leq 1
\right \}
\end{align}
$$
\leq \inf \left\{k \geq 1 :
\frac{\mathbb{E} \left [\Phi \left 
( |f| \cdot \bold{1}_{|f| \leq D}
\right) \right ] }{k}
\leq 1
\right \}
$$
$$
\leq 
\inf \left\{k \geq 1 :
\frac{1}{k} \int_{0}^{D} \Phi^{'}(u) \cdot P(|f| \geq u)du
\leq 1
\right \}
$$
$$
= \max\left (
1, \int_{0}^{D} \Phi^{'}(u) \cdot P(|f| \geq u)du
\right ),
$$
where the first three inequalities above come from the fact that if $A \subset B$, then $\inf A \geq \inf B$, where the third inequality uses equation \ref{firsteq}.
\vskip 0.125in
Set $D = u_0 + (S^2\log n)^{p_1/{2(p_1-r)}}$.  Calculating the value of $\Phi'(u)$ and plugging into equations \ref{eq:2} and \ref{eq:4} it follows that,
$$
I \leq 1 + K(\alpha) \log^{\alpha}(u_0+S^2\log n) \int_0^{D} u \cdot P( |f| \geq u) du.
$$
$$
\leq 1 + K(\alpha)\log^{\alpha}(u_0+S^2 \log n) \int_{0}^\infty u \cdot P( |f| \geq u) du.
$$
$$
=1 + K(\alpha)\log^{\alpha} (u_0+S^2 \log n)\cdot  \|f\|_2^2 
\leq 1 + K(\alpha) \cdot S^2\log^{\alpha}(u_0+S^2 \log n).
$$
The last equality above is from the standard fact that for a random variable, $Y$, $\mathbb{E} |Y|^s$ can be written as 
$s\int_0^{\infty}u^{s-1}P(|Y| \geq u)du$ and noting that $\|f\|_2^2 \leq S^2$.  Indeed, since we assumed that $\|a\|_2 \leq1$ and, by assumption, $(\varphi_i)_{i=1}^n$ are orthogonal, we have that $$\|f\|_2^2 =
\sum_{i \in I} |a_i|^2 \|\varphi_i\|_2^2
\leq \sum_{i \in I} |a_i|^2 \|\varphi_i\|_{p_1}^2 \leq S^2 \sum_{i \in I} |a_i|^2 \leq S^2,
$$ 
where the second inequality above comes from the fact that if $Y$ is a random variable, $\|Y\|_2 \leq \|Y\|_s$ if $2 \leq s$.
\vskip 0.25in
We now obtain a sharper bound on $I$ than that of the warm-up.  Indeed,
\begin{align}
I \leq \inf 
\left \{k \geq 1 :
\mathbb{E} 
\left [
\Phi \left ( \frac{  |f| \cdot \bold{1}_{|f| \leq D}
}{k}
\right ) 
\right ]
\leq 1
\right \}
\end{align}
$$
\leq
\inf 
\left \{k \geq 1 :
\frac{c(\alpha)}{k^2}+
\mathbb{E} 
\left [
\frac{  |f|^2 \cdot \log^{\alpha}\left (\frac{|f|}{k} \right) \bold{1}_{ \left \{|f| \leq D, \frac{|f| \cdot \bold{1}(|f| \leq D)}{k} \geq u_0 \right \}}
}{k^2}
\right ]
\leq 1
\right \}
$$
$$
\leq
\inf 
\left \{k \geq 1 :
\frac{c(\alpha)}{k^2}+
\mathbb{E} 
\left [
\frac{ |f|^2 \cdot \log^{\alpha}(D)}{k^2}
\right ]
\leq 1
\right \}
$$
$$
\leq
\inf 
\left \{k \geq 1 :
\frac{(1+c(\alpha)) \cdot \log^{\alpha}(D)}{k^2}
\leq 1
\right \}
$$
$$
= \max\left (
1, K(\alpha) \cdot \log^{\frac{\alpha}{2}}(D)
\right)
$$
$$
= K(\alpha)\cdot \log^{\frac{\alpha}{2}}(D),
$$
where the third inequality above comes from the fact that if $k \geq 1$ and
$$
\log^{\alpha}\left (\frac{|f|}{k} \right) \bold{1}_{ \left \{|f| \leq D, \frac{|f| \cdot \bold{1}(|f| \leq D)}{k} \geq u_0 \right \}}(\omega) > 0,
$$
then 
$$
D \geq \frac{D}{k} \geq\frac{|f(\omega)|}{k}\geq u_0 \geq e,
$$
and hence,
$$
\log^{\alpha}(D) \geq \log^{\alpha}\left (\frac{|f(\omega)|}{k} \right).
$$
Recalling that $D = u_0 + (S^2\log n)^{p_1/{2(p_1-r)}}$, it follows that
$$
I \leq K(\alpha) \cdot \log^{\frac{\alpha}{2}}(u_0 + S^2 \log(n)).
$$
\vskip 0.125in

We bound $II$.  From equation \ref{eq:probabilitystatemett}, $\|f\|_C \leq K \cdot S \sqrt {\log n}$.  With $E := L^{p_1}(\Omega)$ and $\| \cdot \| := \|\cdot\|_{p_1}$, we invoke Lemma \ref{hahnbanach}.
Thus, there exists $t_1,t_2 \geq 0$ and $w_1$ and $w_2$ such that,
$$
f = t_1 w_1 + t_2 w_2
$$
where 
$t_1 + t_2 =1$ and
$$
\|w_1\|_{p_1} \leq K \cdot S \sqrt{\log n}, \text{  }
w_2 = \sum_{i = 1}^n b_i \varphi_i 
$$
such that
$$
\|b\|_2^2 \leq \frac{K S^2\log n}{C} = K\frac{S^2\log(n)}{S^2\log^{\alpha+2}(n)}
= \frac{K}{\log^{\alpha+1}(n)}.
$$
\vskip 0.25in
We would like to understand and bound the probability measure of $|f|$.  Thus, for $u \geq 0$,
by a union bound,
$$
P(|f| \geq u)
\leq P(|w_1| \geq u) + P(|w_2| \geq u). 
$$
Thus, by Markov's inequality on $|w_1|$ and $|w_2|$ we have,
\begin{align}
\label{eq:5}
P(|f| \geq u) \leq
K \frac{(S^2 \log n)^{p_1/2}}{u^{p_1}} +
\frac{K}{\log^{\alpha+1}(n)u^2}.
\end{align}

\vskip 0.25in
Similar to the warm-up in \ref{eq:4}, using \ref{eq:3} and \ref{eq:2},

\begin{align}
II \leq \inf 
\left \{k \geq 1 :
\mathbb{E} 
\left [
\Phi \left ( \frac{  |f| \cdot \bold{1}_{|f| \geq D}
}{k}
\right ) 
\right ]
\leq 1
\right \}
\end{align}
$$
\leq \inf \left\{k \geq 1 :
\frac{\mathbb{E} \left [\Phi \left 
( |f| \cdot \bold{1}_{|f| \geq D}
\right) \right ] }{k}
\leq 1
\right \}
$$
$$
\leq 
\inf \left\{k \geq 1 :
\frac{1}{k} \Phi(D)P(|f|\geq D)+
\frac{1}{k} \int_{D}^{\infty} \Phi^{'}(u) \cdot P(|f| \geq u)du
\leq 1
\right \}
$$
$$
\leq \max\left (
1, 2\int_{D}^{\infty} \Phi^{'}(u) \cdot P(|f| \geq u)du,
2 \cdot \Phi(D) P(|f| \geq D)
\right ).
$$
Thus,
$$
II \leq 1 + 
2 \int_{D}^{\infty} \Phi^{'}(u) \cdot P(|f| \geq u)du + 2 \Phi(D) \cdot P(|f| \geq D)
$$
$$
\leq 1 +
K(\alpha)\int_{D}^{\infty} u^{r-1}\cdot \frac{(S^2\log n)^{p_1/2}}{u^{p_1}} du
+
K \int_{D}^{\|w_2\|_{\infty}} \frac{\Phi^{'}(u) }{\log^{\alpha+1}(n)u^2} du
+ 2 \Phi(D) \cdot P(|f| \geq D).
$$
$$
= 1+III + IV +V,
$$
where the second inequality comes from equation \ref{eq:5}.

Recalling that $D= u_0+(S^2 \log n)^{\frac{p_1}{2(p_1-r)}}$, we have that
$$
III \leq K(\alpha) (S^2\log(n))^{p_1/2}\frac{1}{D^{p_1-r}} \leq K(\alpha).
$$

\vskip 0.25in
We bound $IV$.  Note that $\|w_2\|_{\infty} \leq n^{1/2}$ since $\|b\|_2 \leq 1$ in the definition of $w_2$.
By equation \ref{eq:2} and since $\|w_2\|_{\infty} \leq n^{1/2}$,
$$
IV \leq K(\alpha) \int_{e}^{n} \frac{\log^{\alpha}(u) u}{ \log^{\alpha+1}(n) u^2} du = 
\frac{K(\alpha)}{\log^{\alpha+1}(n)} \int_{e}^{n} \frac{\log^{\alpha}(u)}{u} du.
$$
$$
\leq \frac{K(\alpha)}{(\alpha+1) \log^{\alpha+1}(n)} \cdot \log^{\alpha+1}(n) = K(\alpha)
$$
\vskip 0.25in
Finally, we bound $V$.  By equation \ref{eq:5}, we have that
$$
V = 2 \cdot \Phi(D) \cdot P(|f| \geq D) 
$$
$$
\leq 2 \cdot D^2 \cdot \log^\alpha(D) \cdot \left(
K \frac{(S^2 \log n)^{p_1/2}}{D^{p_1}} +
\frac{K}{\log^{\alpha+1}(n)D^2}
\right)
$$
$$
\leq  \frac{K\log^\alpha(D) \cdot (S^2\log n)^{p_1/2}}{D^{p_1-2}} +
\frac{K \cdot \log^{\alpha}(D)}{\log^{\alpha+1}(n)}
$$
$$
\leq \frac{K\log^{\alpha}(D) \cdot D^{p_1-r}}{D^{p_1-2}}+K(\alpha)
$$
$$
\leq \frac{K \log^\alpha(D)}{D^{r-2}} + K(\alpha)
$$
$$
\leq K(\alpha),
$$
where the fourth inequality above comes from the definition of $D$ and noting that $S = \max_{1\leq i \leq n} \|\varphi_i\|_{p_1} \leq \max_{1\leq i \leq n} \|\varphi_i\|_{\infty}\ \leq 1$,
and the last inequality comes from the assumption that $r - 2 >0$ is constant.

\vskip 0.25in
Hence, $$II \leq 1 + III+IV +V \leq K(\alpha)$$  Thus, we have
$$
\|f\|_\Phi \leq I + II 
$$
$$
\leq K(\alpha) \left (1+\log^{\frac{\alpha}{2}}(u_0+S^2\log n) \right)
$$
$$
= K(\alpha) \cdot 
\log^{\frac{\alpha}{2}}(u_0+S^2\log n).
$$
\vskip 0.125in
Thus, if $a \in \mathbb{C}^n$, we can set $f/\|a\|_2 := \sum_{i \in I}a_i\varphi_{i}/\|a\|_2$.  Thus,
$$
\| f\|_{\Phi}/\|a\|_2  \leq K(\alpha) \cdot \log^{\frac{\alpha}{2}}(u_0+S^2\log n),
$$
which completes the proof.
\end{proof}

\vskip 0.125in

\subsection{Probabilistic Sharpness Result}
The result of this section uses some ideas similar to those of the probabilistic sharpness result of \cite{Bourgain89_2}.  As before, constants can vary line by line.  
\begin{theorem}
Fix $\rho > 0$ and $\alpha > 0$ and set $n := \frac{m}{2 \rho}m^m$.  Consider the probability space $[0,1]$ with Lebesgue measure.  For $k \in [n]$,
we set 
$$
\varphi_k := \exp(-2 \pi i k (\cdot)),
$$
the orthonormal characters on $[0,1]$.  Let $\left\{ \zeta_k  \right\}_{k=1}^n$ 
be $\{0,1\}$-valued and i.i.d. with mean $\delta = \frac{1}{\log^{\rho}(n)}$.  Let $J_{\omega}$ be the random set,
$$
J_{\omega} := \left \{
k \in [n] : \zeta_k(\omega) = 1
\right \}.
$$
Let $\Phi(u) = u^2 \log^{\alpha}(u)$ for $u \geq u_{0}$, and $\Phi(u) = c(\alpha) u^2$ otherwise.  $u_0 \geq e$ and $c(\alpha)$ are constants, depending on $\alpha$,  chosen so that $\Phi$ is a Young function.  Then,
\begin{equation}
\label{eq:sharpness_main}
\mathbb{E}_J \left (\sup_{\|a\|_2 \leq 1} \left \| \sum_{k \in J}
a_k \varphi_k
\right\|_{\Phi}
\right )
\geq
K(\alpha, \rho)  \log^{\frac{\alpha}{2}}(\log n).
\end{equation}
\end{theorem}
\begin{remark}
We state our main result, Theorem \ref{main}, in terms of a probability over subsets of $[n]$ following a Bernoulli distribution.  We could have, just as easily, stated our main result in terms of an expectation over subsets of $[n]$ as in equation \ref{eq:sharpness_main}.  The latter form, using the expectation, is the formulation of Theorem 1.5 of \cite{Ryou22} (we note that the main technical machinery of Theorem 1.5 is Lemma 4.3, of \cite{Ryou22}, which uses similar techniques to Theorem \ref{main}).  Hence, we can rule out the common randomized technique, subset selection of $[n]$ by a Bernoulli distribution on $[n]$'s indices, to show the existence of a subset which sharpens the bound of equation \ref{eq:main}.  The question of whether the bound of equation \ref{eq:main} is optimal remains open, but equation \ref{eq:sharpness_main} gives evidence that the bound is indeed optimal.  
\end{remark}

\begin{proof}
Throughout the proof, $\log:= \ln$.  By the defintion of $n$, we have that
\begin{equation}
\label{eq:sharpness_5}
m = K(\rho) \frac{\log n}{\log\log n}
\end{equation}
and  
$$
\log(m) = K(\rho) \log \log n,
$$
where $K(\rho) > 0$ is a constant depending only on $\rho$.

\vskip 0.125in

Set $N := \frac{m}{2 \rho}$ and let $x \in \left [0, \frac{1}{8N} \right]$.  We have that
$$
\left |\sum_{k=1}^N
\varphi_k(x)
\right|^2
\geq 
\left |\sum_{k=1}^N
\cos(-2 \pi k x)^2
\right|^2
\geq \left |\sum_{k=1}^N \frac{1}{2}
\right |^2 = \frac{N^2}{4}.
$$
Hence,
\begin{equation}
\label{eq:sharpness_1}
\left |\sum_{k=1}^N
\frac{\varphi_k(x)}{\sqrt{N}} 
\right| \geq 
\frac{\sqrt{N}}{2}.
\end{equation}
By the definition of $n$, we partition $[n]$ into $m^m$ disjoint and contiguous sets of size $N$,
$$
S_t := \left \{
j \in [n] :
N \cdot (t-1) + 1\leq j \leq t \cdot N
\right \}
$$
such that $t \in [m^m]$.  For all $t \in [m^m]$ and $x \in \left [0, \frac{1}{8N} \right]$, we have that
\begin{equation}
\label{eq:sharpness_3}
\left |\sum_{k \in S_t}
\frac{\varphi_k(x)}{\sqrt{N}} 
\right| \geq 
\frac{\sqrt{N}}{2}.
\end{equation}
Indeed,
$$
\left |\sum_{k \in S_t}
\frac{\varphi_k(x)}{\sqrt{N}} 
\right| =
|
\exp[-2 \pi i x \cdot ( - (t-1) \cdot N)]
|
\cdot \left |\sum_{k \in S_t}
\frac{\varphi_k(x)}{\sqrt{N}} 
\right|
$$
$$
=  
\left |\sum_{k=1}^N
\frac{\varphi_k(x)}{\sqrt{N}} 
\right| \geq 
\frac{\sqrt{N}}{2},
$$
where the last inequality comes from equation \ref{eq:sharpness_1}.  

\vskip 0.125in

We will show that 
\begin{equation}
\label{eq:sharpness_2}
p:= \Pr(J \cap S_t = S_t \text{ for some $t \in [m^m] $}) = \Omega(1)
\end{equation}
and prove the result later.  

First, assume that equation \ref{eq:sharpness_2} holds.  Then
$$
\mathbb{E}_J \left (\sup_{\|a\|_2 \leq 1} \left \| \sum_{k \in J}
a_k \varphi_k
\right\|_{\Phi}
\right )
$$
$$
\geq \mathbb{E}_J \left (\sup_{\|a\|_2 \leq 1} \left \| \sum_{k \in J}
a_k \varphi_k
\right\|_{\Phi} \Bigr|
J \cap S_t = S_t \text{ for some $t \in [m^m] $}
\right ) \cdot 
p =: I \cdot p.
$$
Hence, to prove equation \ref{eq:sharpness_main} is true, it suffices to lower bound $I$.  We fix $J \subset [n]$ such that $J \cap S_t = S_t$ for some $t \in [m^m]$.  For simplicity, we set
$$
f:= \sum_{k \in S_t}
\frac{\varphi_k}{\sqrt{|S_k|}}
=
\sum_{k \in S_t}
\frac{\varphi_k}{\sqrt{N}}. 
$$
Then 
$$
\sup_{\|a\|_2 \leq 1} \left \| \sum_{k \in J}
a_k \varphi_k \right \|_{\Phi}
\geq \left \| f \right \|_{\Phi}
$$
$$
= \inf \left \{
w > 0 :
\int_0^1 \Phi 
\left(
\frac{|f|}{w}
\right )
dx\leq 1
\right \}
$$
$$
\geq
\inf \left \{
w > 0 : 
\frac{\int_0^1 \left|f\right |^2
\log^{\alpha} \left ( 
 \frac{|f|}{w}
\right ) 
\bold{1} \left(
\frac{|f|}{w} \geq u_0
\right )
dx
}{w^2} 
\leq 1
\right \}
$$
$$
\geq
\inf \left \{
w > 0 : 
\frac{\int_0^{\frac{1}{8 N}} \left|f\right |^2
\log^{\alpha} \left ( 
 \frac{|f|}{w}
\right ) 
\bold{1} \left(
\frac{|f|}{w} \geq u_0
\right )
dx
}{w^2} 
\leq 1
\right \}
$$
$$
\geq
\inf \left \{
w > 0 : 
\frac{\int_0^{\frac{1}{8 N}} \frac{N}{4}
\log^{\alpha} \left (
\frac{\sqrt{N}}{2w}
\right ) 
\bold{1} \left(\frac{\sqrt{N}}{2w}
\geq u_0
\right )
dx
}{w^2} 
\leq 1
\right \}
$$
\begin{equation}
\label{eq:sharpness_4}
=
\inf \left \{
w > 0 : 
\frac{\frac{1}{32}
\log^{\alpha} \left (
\frac{\sqrt{N}}{2w}
\right ) 
\bold{1} \left(\frac{\sqrt{N}}{2w}
\geq u_0
\right )
}{w^2} 
\leq 1
\right \},
\end{equation}
where the third, fourth, and fifth inequalities come from the fact that if $A \subset B$, then $\inf B \leq \inf A$ and equation \ref{eq:sharpness_3}.

\vskip 0.125in
Suppose that $w \leq \frac{N^{\frac{1}{4}}}{2}$.  Then $\frac{\sqrt{N}}{2 w} 
\geq N^{\frac{1}{4}} \geq u_0$.  Thus, examining $w$ in the set of equation \ref{eq:sharpness_4}, where $0 <w \leq \frac{N^{\frac{1}{4}}}{2}$, we have that
$$
\frac{\frac{1}{32}
\log^{\alpha} \left (
\frac{\sqrt{N}}{2w}
\right ) 
\bold{1} \left(\frac{\sqrt{N}}{2w}
\geq u_0
\right )
}{w^2} 
$$
$$
=
\frac{\frac{1}{32}
\log^{\alpha} \left (
\frac{\sqrt{N}}{2w}
\right ) 
}{w^2} \leq 1,
$$
which implies that 
$$
\sqrt{\frac{1}{32}} 
\log^{\frac{\alpha}{2}} \left (
\frac{\sqrt{N}}{2w} 
\right ) \leq w,
$$
and hence, since $\frac{\sqrt{N}}{2 w} \geq N^{\frac{1}{4}}$,
$$
K(\alpha) \cdot \log^{\frac{\alpha}{2}}
(N) \leq w,
$$
where $K(\alpha) > 0$ is a constant depending only on $\alpha$.  Hence, since $N = \frac{m}{2 \rho}$ and recalling that 
$ 
m = K(\rho) \frac{\log n}{\log \log n}
$,
$$
I \geq \min \left (K(\alpha) \cdot \log^{\frac{\alpha}{2}}(N), \frac{N^{\frac{1}{4}}}{2} \right)
$$
$$
=
K(\alpha) \cdot \log^{\frac{\alpha}{2}}(N)
= 
K(\alpha) \cdot \log^{\frac{\alpha}{2}}\left (\frac{m}{2 \rho} \right)
$$
$$
= K(\alpha, \rho) \log^{\frac{\alpha}{2}}(m) = 
K(\alpha, \rho)  \log^{\frac{\alpha}{2}}(\log n),
$$
where $K(\alpha, \rho) > 0$ is a constant depending only on $\alpha$ and $\rho$. 

Hence,
$$
\mathbb{E}_J \left (\sup_{\|a\|_2 \leq 1} \left \| \sum_{k \in J}
a_k \varphi_k
\right\|_{\Phi}
\right ) \geq I \cdot p 
$$
$$
\geq K(\alpha, \rho)  \log^{\frac{\alpha}{2}}(\log n),
$$
as needed.

\vskip 0.125in
We now will show that 
$$
p:= \Pr(J \cap S_t = S_t \text{ for some $t \in [m^m] $}) = \Omega(1)
$$
to complete the proof.
By independence,
$$
1-p = \prod_{t=1}^{m^m}\Pr(J \cap S_t \not = S_t)
=
\prod_{t=1}^{m^m}(1 - \delta^N)
$$
$$
\leq \exp(- \delta^N m^m).
$$
To complete the proof, we show that 
$$
\delta^N \cdot m^m \geq 1.
$$
Indeed,
$$
\frac{1}{\delta^N}=
\log^{\frac{m}{2}}(n).
$$
Thus, by equation \ref{eq:sharpness_5},
$$
m^m = \left (K(\rho) \frac{\log n}{\log\log n} \right)^m
$$
$$
=
\log^{\frac{m}{2}}(n)
\cdot \left(
K(\rho)\frac{\log^{\frac{1}{2}}(n)}
{\log \log n}
\right)^m
\geq \log(n)^{\frac{m}{2}}= \frac{1}{\delta^N}.
$$
Hence, $\delta^N \cdot m^m \geq 1$.
Thus,
$$
1-p \leq \exp(-\delta^N m^m) \leq
\exp(-1),
$$
which implies that
$$
p \geq 1- \frac{1}{e}.
$$
\end{proof}

\end{document}